\documentclass{amsart}

\usepackage[utf8x]{inputenc}
\usepackage[T1]{fontenc}
\usepackage[english,francais]{babel}

\newtheorem{theorem}{Theorem}
\newtheorem{lemma}[theorem]{Lemma}
\newtheorem{prop}[theorem]{Proposition}
\newtheorem{cor}[theorem]{Corollary}

\newtheorem{remark}{\it Remark\/}

\newcommand{\R}{\mathbb R}

\newcommand{\cF}{\mathcal F}

\newcommand{\cZ}{\mathcal Z}
\newcommand{\ve}{\varepsilon}

\usepackage{amsmath,amsfonts,amssymb,epsfig}
\usepackage{delarray}
\usepackage{graphicx}
\usepackage{url}

\def\latex/{{\protect\LaTeX}}
\def\latexe/{{\protect\LaTeXe}}
\def\bibtex/{{Bib\protect\TeX}}

\newcommand{\C}{\ensuremath{\mathbb C}}
\newcommand{\K}{\ensuremath{\mathbb K}}

\newcommand{\bbR}{\mathbb{R}} 

\newcommand{\cB}{\mathcal{B}}
\newcommand{\cC}{\mathcal{C}}
\newcommand{\cD}{\mathcal{D}}
\newcommand{\FF}{\mathcal{F}}

\newcommand{\cL}{\mathcal{L}}

\begin{document}
\selectlanguage{english}
\title{Deformation of  singular foliations, 1: Local deformation cohomology}
\author{%
Philippe Monnier and Nguyen Tien Zung}
\address{%
Institut de Mathématiques de Toulouse, UMR 5219 CNRS, Université Toulouse III \\
E-mail:  philippe.monnier@math.univ-toulouse.fr, tienzung@math.univ-toulouse.fr } \maketitle
\thispagestyle{empty}
%
\selectlanguage{english}
\begin{abstract}{%
In this paper we introduce the notion of \textit{deformation cohomology} for singular foliations and related objects (namely integrable differential forms and Nambu structures), and study it in the local case, i.e., in the neighborhood of a point. 
}\end{abstract}

\section{Introduction}

This is the first in our series of papers on the problem of deformations of singular foliations
in the sense of Stefan-Sussmann \cite{Stefan-Foliation1974,Sussmann-Foliation1973}. In this paper we will concentrate on the local case, i.e., germs of singular foliations (analytic, smooth or formal), and study the deformation cohomology which governs their infinitesimal deformations. In the subsequent papers, we will discuss the global deformation cohomology, the rigidity problem of singular foliations, and so on. 

In most deformation theories of objects of some given category (e.g., Lie algebras, complex structures, group actions, etc.), one can define a \textit{cohomology group} which controls \textit{infinitesimal deformations}, and other higher cohomology groups which may play the role of obstructions to integrating these infinitesimal deformations into true deformations. One wants to 
do the same thing for singular foliations. In order to do that, one first needs to ``algebraize'' or ``tensorize'' them, 
turn them into objects which can be manipulated with algebraic operations. Our approach to algebraization of singular foliations is via \textit{integrable differentiel forms} and their dual multi-vector fields, called \textit{Nambu structures} (see, e.g., \cite[Chapter 6]{DuZu-PoissonBook2005} and \cite{MiZu-Commuting2014}). 

We note that some authors,
including Androulidalik, Skandalis and Zambon, consider locally finitely-generated involutive modules of vector fields and Lie algebroids as proxies for singular foliations, and obtain many interesting results with this approach, see, e.g., \cite{AndrouZambon-Foliations2016,AndrouSkandalis-Holonomy2009}. Our approach is different from theirs. We believe that our approach is more directly related to the problem of deformations of singular foliations of a given dimension, and that the two approaches are complementary to each other and can be combined for the study of various problems concerning singular foliations.

We refer to \cite[Section 1.5 and Chapter 6]{DuZu-PoissonBook2005} for a brief introduction to singular foliations and some basic results, including the \textit{Stefan-Sussmann theorem} \cite{Stefan-Foliation1974,Sussmann-Foliation1973}, which says that a \textit{distribution} $\cD$
on a manifold $M$ generated by a family $\cC$ of (smooth, analytic or formal) vector fields on $M$ (i.e., at every point $x \in M$ the corresponding tangent subspace $\cD_x$ of $\cD$ is spanned by the vectors $\{ X(x), X\in \cC \}$) 
is \textit{integrable}, i.e., is the \textit{tangent distribution} of a singular foliation à la Stefan-Sussmann, if and only if $\cD$ is invariant with respect to $\cC$, i.e. the local flow of every element of $\cC$ preserves $\cD$. 

In the case when $\dim \cD_x$ does not depend on $x$ then $\cD$ is called a \textit{regular distribution}, and in this case its integrability condition (i.e., $\cD$ is the tangent distribution of a \textit{regular foliation}) is equivalent to the Frobenius \textit{involutivity condition}: the Lie bracket of any two vector fields tangent to $\cD$ is again tangent to $\cD$. In the singular case (when $\dim \cD_x$ is not constant but drops on a subset called the \textit{singular set}), the involutivity condition is still necessary but not sufficient. A simple
counter example is the distribution $\cD$ on $\bbR^2$ given by {$\cD_{(x,y)} = Span(\partial/\partial x)$ if $x \leq 0$ and  $\cD_{(x,y)} = 
Span(\partial/\partial x, \partial/ \partial y)$ if $x > 0$}, which is involutive but not integrable. However, according to a theorem of Hermann \cite{Hermann-Foliation1963}, for locally \textit{finitely generated} disributions (i.e., the family $\cC$ of vector fields which generates $\cD$ can be chosen to be finite, at least locally)
the involutivity condition is sufficient for integrability. To avoid pathologies, we will be mainly interested in singular foliations whose tangent distributions are locally finitely generated.

For regular foliations, the problem of stability (rigidity) was studied by Reeb \cite{Reeb-Foliation1952} and Thurston 
\cite{Thurston-Foliation1974}, among other authors, and a
deformation cohomology (which governs infinitesimal deformations) was defined by Heitsch \cite{Heitsch-Cohomology1975}. We want to extend these theories of stability and infinitesimal deformations of foliations to the case of singular foliations. The motivation is clear: similarly to the fact that 
most functions in practice admit singular points, most foliations that we encounter (e.g., in geometric control theory, sub-Riemannian geometry, dynamical systems, symplectic and Poisson geometry, algebraic geometry, etc.) are in fact \textit{singular}, and many 
interesting things (including global invariants)
are localized at singularities, so one should include singularities in the study.

\section{Tensorization of singular foliations}

\subsection{Integrable differential forms and Nambu structures}

Let us recall that a differential $p$-form $\omega$ on a manifold $M^n$ ($0 \leq q \leq n$) is called \textit{\textbf{integrable}} if it satisfies the following
two conditions for any $(p-1)$-vector field $A$:
\begin{equation}\label{eqn:integrable}
1)\ \omega \wedge i_A \omega = 0 \quad \text{and} \quad
2)\ 
d\omega \wedge i_A \omega = 0.
\end{equation}
In particular, when $p=1$ then the first condition is trivial ($\omega \wedge \omega = 0$ for any 1-form $\omega$), and the second condition is the usual integrability condition for a differential 1-form 
$\omega \wedge d\omega = 0$. If $\omega$ is an integrable $p$-form and $z$ is a regular point of $\omega$, i.e.
$\omega (z) \neq 0$, then in a neighborhood of $z$
there is a local coordinate system $(x_1,\hdots,x_n)$
in which 
\begin{equation}
\omega = f dx_1 \wedge \hdots \wedge dx_{p},
\end{equation}
where $f$ is some function such that $f(z) \neq 0$ . 
The kernel of an integrable $p$-form $\omega$ near a point $z$ where $\omega(z) \neq 0$
is an involutive distribution of corank $p$ which gives a codimension $p$ foliation outside singular points.

A \emph{Nambu structure} of order $q$ on a manifold $M$ is a $p$-vector field $\Lambda$ on $M$ 
such that its contraction 
\begin{equation}
\omega = i_\Lambda \Omega
\end{equation}
with a (local) volume form  $\Omega$
is a (local) differential $p$-form (where $p+q = n$
is the dimension of the manifold). An equivalent 
condition (for smooth or analytic) Nambu structures is as follows: A (smooth or analytic) $q$-vector field $\Lambda$ is a Nambu structure if and only 
if near every point  $x$ such that
$\Lambda(x) \neq 0$ there is a local coordinate system $(x_1,\dots,x_n)$ such that
\begin{equation}\label{eqn:Nambu}
\Lambda = f \frac{\partial}{\partial x_{p+1}}\wedge\ldots\wedge\frac{\partial}{\partial x_n}\,.
\end{equation}

In fact, by a change of coordinates, one can put $f=1$ in Formula \eqref{eqn:Nambu}. This formula cannot be used in the singular formal case, so in order to define a formal singular Nambu structure one has to use Condition \eqref{eqn:integrable} 
(applied to the dual differential form) instead 
(see, e.g.,\cite{DuZu-PoissonBook2005}).

\subsection{From singular foliations to Nambu structures and back}

Given a singular foliation $\cF$, 
we will say that a Nambu structure $\Lambda$ is a \textbf{\textit{tangent Nambu structure}} to a $\cF$ if $$\mathrm{codim}(S(\FF)\setminus S(\Lambda)) \geq 2$$ 
and near each point
$x \notin S(\Lambda)\cup S(\FF)$ there is a local coordinate system in which
$\Lambda = \partial/\partial x_1 \wedge \hdots \wedge \partial/\partial x_p$
and $\cF$ is generated by $\partial/\partial x_1,\dots, \partial/\partial x_p$. Here $S(\Lambda)$ denotes the singular set of $\Lambda$, i.e. the set of points where $\Lambda$ vanishes,
and $S(\FF)$ denotes the set of singular points of $\cF$, i.e., the set of points where the dimension of the tangent distribution drops. If, moreover,  $\mathrm{codim}\big(S(\Lambda)\setminus S(\FF)\big)\geq 2$, 
and $\Lambda$  is without multiplicity in the sense that $\Lambda$ can't be written as  $\Lambda=f^2\Lambda',$  where $f$ is a function which vanishes at the origin, 
then we say that $\Lambda$ is an \textit{\textbf{associated Nambu structure}} to $\cF.$

The above definition works well in the analytic and formal categories, and also in the smooth category under some finiteness conditions: the local associated Nambu structure exists and is unique up to multiplication by a non-vanishing function. (See \cite{MiZu-Commuting2014} for the details). It can be constructed as follows.
Take $p$ local vector fields $X_1,\hdots, X_p$ which are tangent to $\mathcal{F}$ and which are linearly independent almost everywhere, where $p$ is the dimension
of $\mathcal{F}$. Put
\begin{equation}
\Pi=X_1\wedge\hdots\wedge X_q,
\end{equation}
then factorize $\Pi$ as $\Pi=h\Lambda$, where $\mathrm{codim}S(\Lambda)\geq 2.$ If $\mathrm{codim}S(\FF)\geq 2$ then $\Lambda$ is an associated Nambu structure of $\FF$.
If $\mathrm{codim}S(\FF)=1$, then we find a reduced function $s$ such that $S(\FF)=\{s=0\}$, and 
$s\Lambda$ is an associated Nambu structure of $\FF$. 

For example, let $\FF$ be the codimension-1 quadric foliation on  $\mathbb{R}^3$ or $\mathbb{C}^3$ with leaves 
$\{x^2+y^2+z^2=const\}$. Take two tangent vector fields $X=y\frac{\partial}{\partial z}-z\frac{\partial}{\partial y}$, $Y=z\frac{\partial}{\partial x}-x\frac{\partial}{\partial z}$, and put 
$\Pi=X\wedge Y=z\left(x\frac{\partial}{\partial y}\wedge\frac{\partial}{\partial z}+ y\frac{\partial}{\partial z}\wedge\frac{\partial}{\partial x}+z\frac{\partial}{\partial x}\wedge\frac{\partial}{\partial y}\right).$
Then $\Lambda=\dfrac{\Pi}{z}$ is an associated Nambu structure of $\FF$.

Conversely, given a Nambu structure $\Lambda$, consider the set (or the sheaf), denoted by $CIT(\Lambda)$, of all (local) \textit{\textbf{conformally invariant tangent}} (CIT) vector fields of $\Lambda$, i.e. vector fields $X$ satisfying  
\begin{equation}
X\wedge\Lambda=0 \quad \text{and} \quad
\cL_X\Lambda=g\Lambda \;\text{ for some function } g.
\end{equation}
Then one checks easily that $CIT(\Lambda)$ generates an integrable distribution. The corresponding foliation is denoted by $\cF_\Lambda$ and called the \textit{\textbf{associated foliation}} of $\Lambda$.
If $f$ is a non-vanishing function then 
$\cF_{f\Lambda} =\cF_\Lambda$.

The above forward and backward functors give an ``almost 
one to one'' correspondence between local singular foliations and local Nambu structures (up to multiplication by non-vanishing functions) under some mild conditions on the singularities. (See \cite{MiZu-Commuting2014} for precise statements in the holomorphic case). This justifies our use of Nambu structures as a proxy for singular foliations.  

Nambu structures will allow us to study deformations of singular foliations. They also allow us to talk about 
(quasi)homogeneous singular foliations (i.e., foliations associated to linear and/or
(quasi)homogeneous Nambu structures in some 
coordinate system), and study the local normalization problem near a singular point. See \cite{Zung-Nambu2013} and references therein for recent results on the problem of local linearization of singular foliations and of Nambu structures.
Many operations with singular foliations, e.g., pull-back and reduction, can also be done naturally via
associated Nambu structures and integrable differential forms.

Globally, on a manifold, we have a sheaf of local tangent Nambu structures,
which is a locally free module of rank one over the ring of functions. In other words, this sheaf is a line bundle, which is nothing but the 
\textit{anti-canonical line bundle}
of the foliation. Since this line bundle may be twisted and 
does not necessarily admit a global section, we do not necessarily have a global Nambu structure associated to a singular foliation, only local ones.
This will be discussed in more detail in our subsequent paper.

\section{Infinitesimal deformations and deformation cohomologies}

Let $\omega$ be an integrable differential $q$-form 
on a $n$ dimensional manifold $M$. By an \textit{\textbf{infinitesimal deformation}} of $\omega$ we mean a $q$-form $\eta$ such that $\omega+\ve\eta$ is
\textit{integrable modulo $\ve^2$}, where $\ve$
is a formal infinitesimal parameter. In other words,
\begin{equation}
(\omega+ \ve\eta) \wedge i_A (\omega+\ve\eta) \equiv 0 \pmod {\ve^2}\quad \text{and} \quad
d(\omega+ \ve\eta) \wedge i_A (\omega+\ve\eta) \equiv 0 \pmod {\ve^2}
\end{equation}
for any $(q-1)$-vector field $A$. Since $\omega$ is integrable, the above conditions are equivalent to
the following family of linear equations on $\eta$
and $d\eta$:
\begin{equation}
i_A\omega\wedge\eta+i_A\eta\wedge\omega=0 \quad
\text{and} \quad i_A\omega\wedge d\eta+i_A\eta\wedge d\omega=0, \;\forall (q-1)\text{-vector fields }A.
\end{equation} 

If $\eta=\cL_X\omega = i_X d \omega + di_X\omega$, where $X$ is a vector field, then $\eta$ is called a \emph{\textbf{trivial deformation}} of $\omega$ (because it is obtained by the pull-back of $\omega$ with respect to the infinitesimal flow of $X$, i.e. $\omega$ is sent to
$\omega + \varepsilon \eta$  modulo $\ve^2$ by such a flow). Denote by $\cZ(\omega)$ the set 
of infinitesimal deformations of $\omega$, and by $\cB(\omega)$ the set of trivial deformations of $\omega$. It is clear that $\cB(\omega)$ is a vector subspace of $\cZ(\omega)$, and we can define the following quotient vector space, which we denote by 
$DH (\omega)$ and call the \textit{\textbf{deformation cohomology}} of $\omega$:
\begin{equation}
DH (\omega)=\frac{\cZ(\omega)}{\cB(\omega)}.
\end{equation}

Suppose that $\Omega$ is a volume form and $\Lambda$ is a Nambu structure of degree $q$ on $M$. 
The set of \textit{\textbf{infinitesimal deformations}} $\cZ(\Lambda)$ of $\Lambda$ consists of all $q$-vector fields $\Pi$ such that $i_{\Pi}\Omega$ is a infinitesimal deformation of $i_{\Lambda}\Omega$. In other words, $\Pi\in\cZ(\Lambda)$ means that $\Lambda+\ve\Pi$ is a Nambu structure modulo $\ve^2$. If $\Pi=\cL_{X}\Lambda$ for some vector field $X$, then $\Pi$ is called a trivial deformation of $\Lambda$.  We denote by $\cB(\Lambda)$ the set of trivial deformations of $\Lambda$.  The \emph{\textbf{deformation cohomology}} $DH(\Lambda)$
of $\Lambda$ is defined as follows:
\begin{equation}
DH(\Lambda)=\frac{\cZ(\Lambda)}{\cB(\Lambda)}.
\end{equation}

The definition of $\cZ(\Lambda)$ does not depend on the volume form $\Omega$. Usually, $DH(\omega)$ is an infinite dimensional vector space even when $\omega$ is regular. 
For example, if $\omega=dx_1\wedge\ldots\wedge dx_p$,  with $1\leq p\leq n-1$, then $\dim DH(\omega)=+\infty$ because  when a multiform is disturbed it can lose some properties (e.g. the closeness). Here, for any function $f$, the $p$-form $f\omega$ is integrable but not necessarily closed so, does not belong to $\cB(\omega)$. Nevertheless, we have the following proposition.
\begin{prop}
If $\displaystyle \Lambda=\partial/\partial x_1 \wedge \hdots \wedge \partial/\partial x_q$ in $\K^n$, then
$DH(\Lambda)=\{0\}$. 
\end{prop}
\begin{proof}
Consider an infinitesimal deformation $\Lambda + \varepsilon\Pi$ of $\Lambda$. 
Let us first remark that if $\Pi = f \partial/\partial x_1 \wedge \hdots \wedge \partial/\partial x_q$ for some function $f$ then, we have
$\Pi = \cL_{X}\Lambda$ where $X=\big(\int f dx_1\big)\partial/\partial x_1$. 

Therefore, we can assume that in the deformation $\Lambda + \varepsilon\Pi$, the tensor $\Pi$ does not contain a term of type $f\Lambda$.
Consider for $i=1,\hdots, q$ the Hamiltonian vector fields
$$
X_i =(-1)^{p-i} i_{ dx_1\wedge\hdots   
 \widehat{dx_i} \hdots\wedge dx_q } \big( \Lambda + \varepsilon\Pi \big)\,.
$$

We have
$$
X_i = \partial/\partial x_i + \varepsilon \sum_{k=q+1}^n f_k^{(i)} \partial/\partial x_k\,.
$$
Since $\Lambda + \varepsilon\Pi$ is a Nambu tensor modulo $\varepsilon^2$ we have
$X_1\wedge\hdots\wedge X_q= \Lambda + \varepsilon\Pi \, (\mathrm{mod}\, \varepsilon^2)$ and
$\displaystyle [X_i , X_j]=0 \, (\mathrm{mod}\, \varepsilon^2)$.

This last relation then gives $\displaystyle \frac{\partial f_k^{(i)}}{\partial x_j} = \frac{\partial f_k^{(j)}}{\partial x_i}$ for any $i\neq j$ and $k>q$. 
Therefore, by Poincaré's Lemma, there exist functions $F_k$ (for any $k>q$) such that $\displaystyle f_k^{(i)}= \frac{\partial F_k}{\partial x_i}$ for every $i=1,\hdots, q$ and $k>q$.

Now, we put $\displaystyle  X=-\sum_{k=q+1}^n F_k \partial/\partial x_k$. We then have
$$
X_1\wedge\hdots\wedge X_q = \Lambda + \varepsilon \mathcal{L}_X\Lambda \, (\mathrm{mod}\, \varepsilon^2)
$$
and the flow $\phi_X^\varepsilon$ sends $\Lambda$ to $\Lambda + \varepsilon \Pi \, (\mathrm{mod}\, \varepsilon^2)$.

\end{proof}

\subsection{Deformation cohomology of singular foliations}

Let $\Lambda$ be a (germ of a) local (smooth, analytic, or formal) Nambu structure of degree $q$ near the origin in $\K^n$ , and $\cF_\Lambda$ be the foliation generated by $\Lambda$. The set $\cB(\cF_\Lambda)$ of trivial deformations of $\cF_\Lambda$ consists of all (smooth, analyti, or formal) germs of $q$-vector fields $\Pi$
which can be written as  
\begin{equation}
\Pi=\cL_X\Lambda+f\Lambda
\end{equation}
where $X$ is local a vector field and $f$ is a local function neqr the origin in
$\K^n$.
The \emph{\textbf{deformation cohomology}} 
$DH(\cF_\Lambda)$ of the foliation $\cF_\Lambda$ is defined as follows:  
\begin{equation}
DH(\cF_\Lambda)=\frac{\cZ(\Lambda)}{\cB(\cF_\Lambda)}.
\end{equation}
Similarly, if $\cF_\omega$ is generated by a local integrable $p$-form $\omega$, then the set of trivial deformations 
$\cB(\cF_\omega)$ consists of all $p$-forms of the type $\cL_X\omega+f\omega$, where $X$ denotes a local vector field
and $f$ denotes a local function. The \textbf{\textit{deformation cohomology}} of $\cF_\omega$ can then be defined as follows:
\begin{equation}
DH(\cF_\omega)=\frac{\cZ(\omega)}{\cB(\cF_\omega)}.
\end{equation}
The following lemma, whose proof is straightforward, says that the cohomology of singular foliations doesn't depend on the choice of its associated Nambu structures or integrable forms:  

\begin{lemma} Let $\Lambda$ be a (local or global) Nambu structure on a manifold. Suppose that $\Omega$ is a (local or global) volume form and $u$ is an invertible function on the manifold. If $\omega=i_\Lambda\Omega$ then
\begin{equation}
DH(\cF_\Lambda) \cong DH(\cF_\omega)\quad \text{and} \quad DH(\cF_\Lambda) \cong DH(\cF_{u\Lambda}).
\end{equation}
\label{lem:isomfonctionpres}
\end{lemma}


\subsection{Nambu structures of order 0 (functions)}

Suppose that  $\Lambda=f$ is a (smooth or analytic) function (i.e. a $0$-vector field) in a neighborhood of 0 in $\K^n$ 
($\K$ is $\R$ or $\C$). 

We denote by $\mathcal{O}_n$ the vector space of germs at 0 of (smooth or analytic) functions on $\K^n$ 
 and $\mathfrak{X}(\K^n)$ the vector space of germs at 0 of (smooth, analytic) vector fields. 
\begin{theorem}
With the notations above, we have
\begin{align}
DH(f)&=\frac{\mathcal{O}_n}{\{X(f)\, | \, X\in\mathfrak{X}(\K^n)\}}=\frac{\mathcal{O}_n}{\left\langle\frac{\partial f}{\partial_{x_1}},\ldots,\frac{\partial f}{\partial_{x_n}}\right\rangle},\\
DH(\cF_f)&=\frac{\mathcal{O}_n}{\{X(f)+ cf\, | \, X\in\mathfrak{X}(\K^n),c\in\mathcal{O}_n\}}=\frac{\mathcal{O}_n}{\left\langle f,\frac{\partial f}{\partial_{x_1}},\ldots,\frac{\partial f}{\partial_{x_n}}\right\rangle}.
\end{align}
In particular, $\mathrm{dim} DH(f)=\mu(f)$ (\textit{the so-called Milnor number}) and $\mathrm{dim} DH(\cF_f)=\tau(f)$ (\textit{the so-called Tjurina number}).\\
\end{theorem}

\begin{proof}
It is obvious. The set of infinitesimal deformations of $f$ is just $\mathcal{O}_n$.
\end{proof}

Note that in this case, if $\omega=fdx_1\wedge\hdots\wedge dx_n$ then $DH(\omega)$ is the quotient of $\Omega^n(\K^n)$ by 
$\{ d(f\theta)\,|\, \theta\in \Omega^{n-1}(\K^n) \}$ (denoting by $\Omega^{k}(\K^n)$ the vector space of $k$-differential forms) which is 
isomorphic to the quotient of $\mathcal{O}_n$ by  $\{X(f) + (div X)f\, | \, X\in\mathfrak{X}(\K^n) \}$. 
This deformation space has been computed in \cite{Monnier-Nambucohomology2001} (Theorem 3.14) when $f$ is a quasihomogeneous polynomial with 
an isolated singularity at 0 (its dimension is the Milnor number of $f$).

\subsection{Top order multi-vector fields}
Assume that
\begin{equation}\Lambda=f\frac{\partial}{\partial x_1}\wedge\ldots\wedge\frac{\partial}{\partial x_n}
\end{equation}
where $f$ is either a smooth real function or a real or complex analytic function such that $f(0) = 0$ and moreover $0$ is a singular point of $f$, i.e. $df(0) = 0$.

\begin{theorem}
With the same notations as above, we have
\begin{equation}
DH(\FF_\Lambda) \cong \frac{\mathcal{O}_n}
{\left\langle f,\frac{\partial f}{\partial_{x_1}},\ldots,
\frac{\partial f}{\partial_{x_n}}\right\rangle}
\end{equation}
and $\displaystyle \dim H_{def}(\FF_\Lambda)=\tau(f)$ is the  Tjurina number of $f$ at 0.
Moreover, 
\begin{equation}
DH(\Lambda) \cong \frac{\mathcal{O}_n} {\{X(f)-(div X)f\, | \, X\in\mathfrak{X}(\K^n) \}}\,.
\end{equation}
\end{theorem}

\begin{proof}
The vector space of infinitesimal deformations of $\Lambda$ is the space of (germs of) $n$-vector fields of type 
$f\frac{\partial}{\partial x_1}\wedge\ldots\wedge\frac{\partial}{\partial x_n}$ with $g\in \mathcal{O}_n$ which is isomorphic to $\mathcal{O}_n$. 
If $X$ is a vector field, we have $\displaystyle \mathcal{L}_X\Lambda = \big( X(f)-(div X)f \big) \frac{\partial}{\partial x_1}\wedge\ldots\wedge\frac{\partial}{\partial x_n}$ which gives the expression of $\mathcal{B}(\Lambda)$ and $\cB(\cF_\Lambda)$.
Finally, one easily checks that 
$\displaystyle \big\{ X(f)-(div X)f + gf \, | \, X\in\mathfrak{X}(\K^n),\, g\in \mathcal{O}_n \big\}$ is 
$\left\langle f,\frac{\partial f}{\partial_{x_1}},\ldots,
\frac{\partial f}{\partial_{x_n}}\right\rangle$.

\end{proof}

One can find some computations of this cohomology space in the case where $f$ is a quasihomogeneous polynomial with an isolated singularity at 0 in \cite{Monnier-Nambucohomology2001} and  \cite{Monnier-Poissoncohomology2002}.
More precisely, if $n=2$, it is the (germified) Poisson cohomology of the Poisson structure $\Lambda$ 
(see Theorems 4.9 and 4.11 in \cite{Monnier-Poissoncohomology2002}).
If $n\geq 3$, it is related to a Nambu cohomology space associated to $\Lambda$, denoted by $H_{f,n-2}^n(\K^n)$ or $H_\Lambda^2(\K^n)$ in 
\cite{Monnier-Nambucohomology2001} (Corollary 3.19). In these two cases, the dimension of the deformation cohomology space is finite and depends on the Milnor number
of $f$.


\subsection{Decomposable integrable forms with small singularities}

In this section, we work on $\C^n$, in the complex analytic category.
Suppose that $\Lambda$ is an analytic Nambu structure in a neighborhood of 0 in $\C^n$ and $\omega=i_{\Lambda}\Omega$, $\Omega$ is a volume form.
If $\omega$ is decomposable (i.e. $\omega=\omega_1\wedge\ldots\wedge\omega_p$) and 
$\mathrm{codim}(\omega)\geq 3$ then by Malgrange (see \cite{Mal}): 
\begin{equation}
\omega=ud{f_1}\wedge\ldots\wedge d{f_p},
\end{equation}
where $u$ is a function with $u(0)\neq 0$.
According to Lemma \ref{lem:isomfonctionpres} we can assume that $u=1$.

\begin{prop}
Let $\omega=d{f_1}\wedge\ldots\wedge d{f_p}$ be a complex analytic integrable $p$-form and $\eta$  is an infinitesimal deformation $\omega$. If $\mathrm{codim}S(\omega)\geq p+2$ then
$$\eta=a_0df_1\wedge\ldots\wedge df_p+\sum_{i=1}^p df_1\wedge\ldots\wedge df_{i-1}\wedge da_i\wedge df_{i+1}\wedge\ldots\wedge df_p.$$
\label{prop:decompositiondefinfinitesimal}
\end{prop}
It means that $\omega + \epsilon \eta$ is also decomposable and admits first integrals modulo $\epsilon^{2}$.

\begin{proof}
By definition, $\eta$ satisfies for all $(p-1)$-vector field $A$ :
\begin{eqnarray}
i_A\eta \wedge df_1\wedge\hdots\wedge df_p + i_A(df_1\wedge\hdots\wedge df_p)\wedge \eta & = & 0 \label{eqn:infdefdecomp1}\\
i_A(df_1\wedge\hdots\wedge df_p) \wedge d\eta & = & 0 \label{eqn:infdefdecomp2}
\end{eqnarray}
 We first claim that (\ref{eqn:infdefdecomp2}) is equivalent to 
\begin{equation}
    df_i\wedge d\eta=0 \quad (\forall i=1,\hdots,p)\,. 
\label{eqn:infdefdecomp3}
\end{equation}
Indeed, if $x\notin S(\omega)$ then $df_1(x),\hdots, df_p(x)$ are independent and if $E_x$ is the subspace of $(\C^n)^\ast$ generated by the linear forms $df_1(x),\hdots, df_p(x)$, we consider constant vector fields $X_1,\hdots,X_p$ such that 
$\langle df_i(x),X_j(x) \rangle = \delta_{ij}$ (Kronecker symbol) for all $i$ and $j$.
We put $A_i=X_1\wedge\hdots\widehat{X_i}\hdots\wedge X_{p}$ and 
if (\ref{eqn:infdefdecomp2}) is satisfied, it gives $df_i(x)\wedge d\eta(x)=0$ for all $i=1,\hdots,p$. The converse is obvious.

Now, using successively the vanishing of the relative de Rham cohomology spaces $H^p(\Omega^\ast_{f_1,\hdots,f_k})$ for $k=1,\hdots,p$
(see \cite{Mal}) we get
\begin{equation}
    \eta = d\theta + c df_1\wedge\hdots\wedge df_p\,,
\label{eqn:infdefdecomp4}
\end{equation}
for some $(p-1)$-form $\theta$ and function $c$.

In the same way as above, we can show that (\ref{eqn:infdefdecomp1}) implies 
\begin{equation}
    df_i\wedge df_j\wedge \eta = 0\quad \forall i,j=1\hdots,p\,.
\end{equation}
Consequently, in the decomposition (\ref{eqn:infdefdecomp4}), $d\theta$ satisfies this condition too. 
For every $i$, using successively the division theorem (see for instance Proposition 1.1 in \cite{Mal}),
one can show easily that 
\begin{equation}
    df_i\wedge d\theta = df_1\wedge\hdots\wedge df_p\wedge\beta_i
\end{equation}
where $\beta_i$ is a 1-form. Now, we get 
$\displaystyle df_1\wedge\hdots\wedge df_p\wedge d\beta_i = 0$ so $\beta_i$ is a 1-cocyle in the relative de Rham cohomology 
$H^1(\Omega^\ast_{f_1,\hdots,f_p})$, which gives $\beta_i = da_i + \sum_{j=1}^pb_{ij}df_j$ for some functions $a_i$ and $b_{ij}$. Therefore, 
\begin{equation}
     df_i\wedge d\theta = df_1\wedge\hdots\wedge df_p\wedge da_i \,.
\end{equation}
It gives
\begin{equation}
    df_i\wedge \Big( d\theta  + \sum_{j=1}^p(-1)^j df_1\wedge\hdots\wedge\widehat{df_j}\wedge\hdots\wedge df_p\wedge da_j\Big) = 0
\end{equation}
for every $i=1,\hdots,p$, which implies, by the division theorem,
\begin{equation}
    d\theta  + \sum_{j=1}^p(-1)^j df_1\wedge\hdots\wedge\widehat{df_j}\wedge\hdots\wedge df_p\wedge da_j = b df_1\wedge\hdots\wedge df_p
\end{equation}
for some function $b$. The proposition follows.
\end{proof}

\bigskip

If, $\omega=df_1\wedge\hdots\wedge df_p$, we consider $F$ the analytic map from $\C^n$ to $\C^p$ defined by $F=(f_1,\hdots, f_p)$.
If $X\in\mathfrak{X}(\C^n)$ and $H=(H_1,\hdots,H_p)$ is an analytic map from $\C^p$ to $\C^p$, we denote 
$X.F=(X.f_1,\hdots,X.f_p)$ and $H(F)=(H_1(f_1,\hdots,f_p),\hdots,H_p(f_1,\hdots,f_p))$. Now, we put
$$
\mathcal{I}_{RL}(F) = \big\{X.F + H(F)\,|\, X\in\mathfrak{X}(\C^n)\,,\, H\in (\mathcal{O}_p)^p\big\}
\quad \mathrm{and}\quad \mathcal{Q}_{RL}(F) = \frac{(\mathcal{O}_n)^p}{\mathcal{I}_{RL}(F)}\,.
$$
Recall that $\mathcal{Q}_{RL}(F)$ measures the stability of the germ $F$ and the versal deformations of $F$ with respect to the Right-Left-equivalence (see for instance \cite{Arnold}).
More precisely, another germ of analytic map $G$ is RL-equivalent to $F$ if there exists a germ of analytic diffeomorphism $\phi$ of 
$(\C^n,0)$ and a germ of analytic diffeomorphism $\psi$ of $(\C^p,0)$ such that $G=\psi\circ F\circ\phi$.

\begin{theorem}
If $\omega=df_1\wedge\hdots\wedge df_p$ with $\mathrm{codim}S(\omega)\geq p+2$, then, with the notations above, we have
\begin{equation}
DH(\cF_{\omega}) \simeq  \mathcal{Q}_{RL}(F)\,.
\end{equation}
\end{theorem}
\begin{proof}
We denote 
\begin{eqnarray*}
A & = & \big\{ \sum_{i=1}^p df_1\wedge\ldots\wedge df_{i-1}\wedge da_i\wedge df_{i+1}\wedge\ldots\wedge df_p\,|\, a_i\in\mathcal{O}_n \big\} \\
B & = & \big\{ bdf_1\wedge\hdots\wedge df_p \,|\, b\in\mathcal{O}_n \big\} \\
C & = & \big\{ \sum_{i=1}^p df_1\wedge\ldots\wedge df_{i-1}\wedge dX(f_i)\wedge df_{i+1}\wedge\ldots\wedge df_p  \,|\, X\in\mathfrak{X}(\C^n) \big\}
\end{eqnarray*}

By Proposition \ref{prop:decompositiondefinfinitesimal}, we have 
$$
DH(\cF_{\omega}) = \frac{A+B}{B+C} \simeq \frac{A}{A\cap(B+C)}\,.
$$

Clearly, $C$ is included in $A$. If $bdf_1\wedge\hdots\wedge df_p\in B$ is in $A$, i.e. of the form
$\sum_{i=1}^p df_1\wedge\ldots\wedge df_{i-1}\wedge da_i\wedge df_{i+1}\wedge\ldots\wedge df_p$ then, we have 
$db\wedge df_1\wedge\hdots\wedge df_p=0$ which gives $b=H(f_1,\hdots,f_p)$ where $H\in\mathcal{O}_p$ (see \cite{Mal}, Theorem 2.1.1). 
Conversely, if $b=H(f_1,\hdots,f_p)$ with $H\in\mathcal{O}_p$ then $bdf_1\wedge\hdots\wedge df_p$ is in $A\cap B$.

Therefore, we have
$$
DH(\cF_{\omega}) \simeq \frac{A}{D+C}\,,
$$
where $D=\big\{ H(f_1,\hdots,f_p)df_1\wedge\hdots\wedge df_p \,|\, H\in\mathcal{O}_p \big\}$.

Now, we define $\Phi : (\mathcal{O}_n)^p \longrightarrow DH(\cF_{\omega})$ such that if $G=(g_1,\hdots,g_p)$ we have
$$
\Phi(G)= \big[ \sum_{i=1}^p df_1\wedge\ldots\wedge df_{i-1}\wedge dg_i\wedge df_{i+1}\wedge\ldots\wedge df_p \big]\,.
$$
It is a surjective linear map. It is clear that $\mathcal{I}_{RL}(F)$ is included in the kernel of $\Phi$. Moreover,
if $\Phi(G)=0$ then there exist a vector field $X$ on $\C^n$ and an analytic map $K$ from $\C^p$ to $\C^p$ such that
$$
\sum_{i=1}^p df_1\wedge\ldots\wedge df_{i-1}\wedge (dg_i - X.f_i)\wedge df_{i+1}\wedge\ldots\wedge df_p =
K(f_1,\hdots,f_p)df_1\wedge\hdots\wedge df_p\,.
$$
It implies that for every $i=1,\hdots,p$, we have $df_1\wedge\hdots\wedge df_p\wedge(dg_i - X.f_i)=0$ which gives (see \cite {Mal})
that $g_i=X.f_i + H_i(f_1,\hdots,f_p)$ for some $H_i\in\mathcal{O}_p$. Therefore, $G\in \mathcal{I}_{RL}(F)$.
\end{proof}

\begin{cor}
If $\omega=df$ (Nambu structure of order $n-1$) and $\mathrm{codim}S(df)\geq 3$ then 
\begin{equation}DH(\cF_{df}) \simeq \frac{\mathcal{O}_n}{\left\{ a_i\frac{\partial f}{\partial_{x_1}}+\ldots+a_n\frac{\partial f}{\partial_{x_n}}+h\circ f\, | \, a_i\in\mathcal{O}_n, h\in\mathcal{O}_1\right\}}
\end{equation}
In particular, $\mu(f)\geq \mathrm{dim} DH(\cF_{df})\geq\tau(f)-1$.
\end{cor}

\begin{remark}
We consider the two ideals of $\mathcal{O}_n$, 
$I_f=\left\langle\frac{\partial f}{\partial_{x_1}},\ldots,\frac{\partial f}{\partial_{x_n}}\right\rangle$ 
and $J_f=\left\langle f,\frac{\partial f}{\partial_{x_1}},\ldots,\frac{\partial f}{\partial_{x_n}}\right\rangle$, 
If we assume that $df(0)=0$ then $X.f$ is not a constant for any vector field $X$ so, we have
$I_f\oplus\C \subset \mathcal{I}_{RL}(f)\subset J_f\oplus\C$ which gives $\mu(f)-1\geq \mathrm{dim} DH(\cF_{df})\geq\tau(f)-1$.
Consequently, if 0 is an isolated singularity of $f$, i.e. $\mu(f) < \infty$, then $\mathrm{dim} DH(\cF_{df})<\infty$.
If moreover, $f$ is a quasihomogeneous polynomial, then $I_f=J_f$  which gives $\mathrm{dim} DH(\cF_{df})= \mu(f)-1 = \tau(f)-1$.
\end{remark}

\begin{cor}
If $0$ is an isolated singularity of $\omega=udf_1\wedge\ldots\wedge df_p$  then  $\mathrm{dim} DH(\cF_{\omega})<\infty$
\end{cor}
\begin{proof}
We prove that, denoting $F=(f_1,\hdots, f_p)$, the quotient $\displaystyle \frac{(\mathcal{O}_n)^p}{\big\{X.F\,|\, X\in\mathfrak{X}(\C^n)\big\}}$ has a finite dimension. If we denote by $\mathfrak{M}$ the ideal of $\mathcal{O}_n$ formed by the functions vanishing at 0, we prove that there is a positive integer  $N$ such that if $g_i\in \mathfrak{M}^N$, $i=1,\hdots ,p$, then there exists a vector field $X$ such that $X.f_i=g_i$, $i=1,\hdots ,p$. The corollary will follow naturally.

We prove it by induction on $p$. The case $p=1$ is a direct consequence of Hilbert's Nullstellensatz Theorem. If the statement is true for $p-1$, we prove the existence of an integer $N$ such that for all $g\in \mathfrak{M}$, there is a vector field $X$ which satisfies $X.f_p=g$ and $X.f_i=0$,  $i=1,\hdots ,p-1$.

We consider the ideal $\mathcal{I}$ of $\mathcal{O}_n$ formed by functions  
$g\in \mathfrak{M}$ such that there exists a vector field $X$  satisfying $X.f_p=g$ and $X.f_i=0$,  $i=1,\hdots ,p-1$.

For $1\leq i_1 < \hdots < i_p \leq n$, denoting 
$$
X= \sum_{j=1}^p (-1)^{p+j} det\Big( \frac{\partial(f_1,\hdots,f_{p-1})}{\partial(x_{i_1},\hdots,\widehat{x_{i_j}},\hdots, x_{i_p})} \Big) \frac{\partial}{\partial x_{i_j}}
$$
we have  $X.f_i=0$,  $i=1,\hdots ,p-1$ and $X.f_p= det\Big( \frac{\partial(f_1,\hdots,f_{p})}{\partial(x_{i_1},\hdots, x_{i_p})} \Big)$.
Therefore, the function $\displaystyle det\Big( \frac{\partial(f_1,\hdots,f_{p})}{\partial(x_{i_1},\hdots, x_{i_p})} \Big)$ is in $\mathcal{I}$.

By the hypothesis, the zero locus of $\mathcal{I}$ is $\{0\}$ and it finishes the proof, using Hilbert's Nullstellensatz Theorem.
\end{proof}

\subsection{Vector fields and linear Nambu structures}

If $\Lambda = X$ is a vector field, the leaves of the associated foliation are integral curves of $X$. The normalization
of this foliation is the same as the
orbital normalization of $X$.

In the case $\Lambda=X$ is a formal vector field in $(\K^n,O)$ whose linear part  $X^{(1)}$ is non-trivial, we denote by $DH_{lin}(X^{(1)})$
the quotient of the vector space of {\it linear} infinitesimal deformations of $X^{(1)}$ by the vector space of {\it linear} trivial 
deformations of $X^{(1)}$.

Then by the classical Poncar\'e-Dulac formal normalization theory we have the following proposition.

\begin{prop}
Suppose that $X^{(1)}$ is non-resonant, then we have :
\begin{itemize}
    \item[(i)]   $DH(X) = DH_{lin}(X^{(1)})$ and $\dim DH(\cF_X) = \dim DH_{lin}(X^{(1)}) -1$.
    \item[(ii)]  $\dim DH_{lin}(X^{(1)}) = n^2 -d \geq n$ where $d$ is the dimension of the adjoint orbit of $X^{(1)}$ in $\mathfrak{gl}_n(\K)$.
    \item[(iii)] In the generic case (i.e. the eigenvalues of $X^{(1)}$ are distincts) we have $\dim DH_{lin}(X^{(1)}) =~n$.
\end{itemize}
\end{prop}

\begin{proof}
By hypothesis, the vector field $X$ is formally linearizable, we then can assume that it coincides with its linear part $X^{(1)}$.
Moreover, if $X^{(1)} + \varepsilon Y$ is an infinitesimal deformation of $X^{(1)}$ then $Y$ may be written as $Y=Y^{(1)} + \widetilde{Y}$ where $\widetilde{Y}$ contains only terms of degree larger or equal to 2. Since $X^{(1)}$ is non-resonant, we have 
$\widetilde{Y} = [X^{(1)} , \widetilde{Z}]$ for some formal vector field $\widetilde{Z}$ which contains only terms of degree larger or equal to 2. 
Finally, $DH(X)$ is the quotient of the vector space of linear vector fields by the vector space of vector fields of type $[X^{(1)} , Z]$
where $Z$ is a linear vector field (whose dimension is the dimension of the adjoint orbit of $X^{(1)}$).
In the same way, $DH(\cF_X)$ is the quotient of the vector space of linear vector fields by the vector space of vector fields of type
$[X^{(1)} , Z] + \lambda X^{(1)}$ where $Z$ is a linear vector field and $\lambda\in\K$.

Finally, recall that the dimension of the adjoint orbit of $X^{(1)}$ is less than $n(n-1)$ and if the eigenvalues of $X^{(1)}$ are distincts,
it is exactly $n(n-1)$.
\end{proof}

Note that the point {\it (iii)} of this proposition can be true even if $X^{(1)}$  has eigenvalues of multiplicity strictly larger than 1. 
It is the case if in the Jordan decomposition of $X^{(1)}$, there is only one Jordan block corresponding to each eigenvalue. 

In the resonant case, the formal deformation cohomology can be infinite-dimensional. \\

Let us now recall that there are two types
of linear Nambu structures:

\underline{Type 1}:  
$\Lambda$  is dual to a decomposable linear integrable differential form  $\omega = dx_1 \wedge \dots 
\wedge dx_{p-1} \wedge dQ,$ where $Q$ is a quadratic function.

\underline{Type 2}: $  \Lambda$ is decomposable:      
$  \Lambda = \partial / \partial x_{1} \wedge ... \wedge \partial / \partial x_{q-1} \wedge                 (\sum_{i,j=q}^{n} b^i_j x_i \partial / \partial x_j).             
$   

It has been shown in \cite{DuZu-Nambu1999, Zung-Nambu2013}
that linear Nambu structures of Type 1 with a nondegenerate quadratic function $Q$ in its formula are formally and analytically rigid (and they are also smoothly rigid if $Q$ satisfies a natural condition on its signature).  In fact, the proofs in these papers also dealt with deformation cohomology, so we can conclude that
the formal and analytic deformation cohomology of a linear Nambu structure of Type 1 is trivial if the quadratic function $Q$ in its formula is nondegenerate. If, moreover, the signature of $Q$ is different from $(1,*)$
then the local smooth deformation cohomology is also trivial. 

As regards linear Nambu structure of Type 2, the situation is similar to that of linear vector fields 
$X = \sum_{i,j=q}^{n} b^i_j x_i \partial / \partial x_j$ 
in the formula. In particular, if $X$ is non-resonant then
$  \Lambda = \partial / \partial x_{1} \wedge ... \wedge \partial / \partial x_{q-1} \wedge X$ has trivial formal deformation cohomology. 
(See \cite{DuZu-Nambu1999, Zung-Nambu2013} for
the details.)

\end{document}